\documentclass[12pt]{amsart}
\usepackage{latexsym,enumerate}
\usepackage{amssymb, xcolor,mathrsfs}
\usepackage{amsmath,amsthm,amsfonts,amssymb,latexsym, mathabx}
\usepackage{booktabs}
\usepackage{makecell}
\usepackage{listings}
\usepackage{xcolor}

\lstdefinelanguage{GAP}{
	morekeywords={if, then, elif, else, fi, for, while, do, od, function, local, return, end, and, or, not, in},
	sensitive=true,
	morecomment=[l]{\#},
	morestring=[b]",
	morestring=[b]',
}

\usepackage{comment}
\usepackage[
backend=biber,
style=numeric,      % numeric citation style (in-text = numbers)
sorting=nyt,        % ordering in bibliography (change if you want)
maxnames=10,
giveninits=true,    % show initials like I. M. Isaacs
uniquename=false,
doi=false,
isbn=false,
url=true,
]{biblatex}
\newtheorem{theorem}{Theorem}[section]
\newtheorem{lemma}[theorem]{Lemma}
\newtheorem{proposition}[theorem]{Proposition}

\newtheorem{question}[theorem]{Question}

\theoremstyle{definition}
\newtheorem{definition}[theorem]{Definition}

\newtheorem*{thmA}{Theorem A}
\newtheorem*{thmB}{Theorem B}

\newcommand{\Z}{\mathbb{Z}}

\newcommand{\F}{\mathbb{F}}

\newcommand{\la}{\langle}
\newcommand{\ra}{\rangle}

\DeclareMathOperator{\GL}{GL}

\DeclareMathOperator{\PSU}{PSU}
\DeclareMathOperator{\AGL}{AGL}
\DeclareMathOperator{\SO}{SO}

\DeclareMathOperator{\Soc}{Soc}

\newcommand{\Cent}{\mathbf{Z}}

\makeatletter
\newcommand{\hathat}[1]{% 
	\begingroup%
	\let\macc@kerna\z@%
	\let\macc@kernb\z@%
	\let\macc@nucleus\@empty%
	\hat{\raisebox{.35ex}{\vphantom{\ensuremath{#1}}}\smash{\hat{#1}}}%
	\endgroup%
}
\makeatother

\DeclareCiteCommand{\tabcite}%[\mkbibbrackets]
{\usebibmacro{cite:init}%
	\usebibmacro{prenote}}
{\usebibmacro{citeindex}%
	\usebibmacro{cite:comp}}
{}
{\usebibmacro{cite:dump}%
	\usebibmacro{postnote}}

\DefineBibliographyStrings{english}{%
	bibliography = {BIBLIOGRAPHY},
}

\usepackage{setspace}
\AtBeginBibliography{\singlespacing}
\DeclareFieldFormat{labelnumber}{#1}

\DeclareFieldFormat[article]{year}{(#1)}
\DeclareFieldFormat[book]{year}{#1}

\DeclareFieldFormat[article,book,inproceedings,thesis,unpublished,misc]{title}{#1}

\DeclareFieldFormat[article]{volume}{\mkbibbold{#1}}
\DeclareFieldFormat{journaltitle}{\mkbibemph{#1}}
\DeclareFieldFormat[book]{title}{\mkbibemph{#1}}
\DeclareFieldFormat{pages}{#1}
\DeclareFieldFormat{doi}{#1}
\DeclareFieldFormat[online]{title}{#1}
\DeclareFieldFormat{howpublished}{#1}
\DeclareFieldFormat{eprint}{preprint, arXiv:#1}
\DeclareFieldFormat{archivePrefix}{#1}

\renewbibmacro{in:}{}

\DeclareBibliographyDriver{article}{%
	\printnames{author}%
	\newunit\newblock
	\printfield{title}% 
	\newunit\newblock
	\printfield{journaltitle}%
	\setunit{\addspace}%
	\printfield{volume}%
	\iffieldundef{number}{}{%
		(\printfield{number})%
	}%
	\setunit{\addcomma\addspace}%
	\iffieldundef{pages}{\addcomma}{
		\printfield{pages}%
	}
	\printfield{year}%
	\finentry
}

\DeclareBibliographyDriver{book}{%
	\printnames{author}%
	\newunit\newblock
	\printfield{title}%
	\iffieldundef{edition}{}{%
		\setunit{\addcomma\space}%
		\printfield[edition]{edition}%
		,
	}%
	\finentry
	\setunit{\space}
	\newunit\newblock
	\printlist{publisher}%
	\setunit{\addcomma\space}%
	\printfield{year}%
	.% Only print edition if it exists
	
}

\DeclareBibliographyDriver{misc}{%
	\printnames{author}%
	\newunit\newblock
	\printfield{title}%
	\setunit{\addcomma\space}%
	\iffieldundef{note}{}{%
		\printfield{eprint}
		\addcomma
	}%
	\addspace
	\printfield{year}%
	\finentry
}

\DeclareBibliographyDriver{unpublished}{%
	\printnames{author}%
	\newunit\newblock
	\printfield{title}%
	\iffieldundef{year}
	{}
	{%
		\setunit{\addcomma\space}%
		\printfield{year}%
	}%
	\finentry
}

\DeclareBibliographyDriver{online}{%
	\printnames{author}%
	\newunit\newblock
	\printfield{title}% 
	\newunit\newblock
	\printfield{howpublished}
	\setunit{\addcomma\space}
	\printfield{year}%
	\newunit\newblock
	\printfield{url}
	\finentry
}

\begin{document}
	
	\title[$B$-groups of order $p^4$]{A Determination of $B$-groups of Order $p^4$}

	\author[Christopher Herbig]{Christopher Herbig}
	\address{%
		Little Priest Tribal College\\
		Winnebago, NE 68071 \\
		USA}
	\email{christopher.herbig@littlepriest.edu}

	

	\date{\today}
	
	\maketitle
	
	\begin{abstract}
		A group $H$ is said to be a $B$-group if whenever a primitive permutation group $G$ contains a regular subgroup isomorphic to $H$, then $G$ is doubly transitive. Going as far back as William Burnside, several authors have investigated whether certain families of groups are $B$-groups. Using the O'Nan-Scott Theorem as our starting point, we classify the $B$-groups of order $p^4$ where $p$ is a prime.
	\end{abstract}
	
	%\tableofcontents
	
	%%%%%%%%%%%%%%%%%%%%%%%%%%%%%%%%%%%%%%%%%%%%%%%%%%%%%%%%%%%%%%%%%%%%%%
	\section{Introduction}
	
	We begin by introducing some notation. All groups considered are assumed to be finite. We will let $C_n$, $Q_n$, and $D_n$  denote respectively the cyclic, generalized quaternion, and dihedral groups of order $n$. We take $H_n(p)$ to denote the group of $n \times n$ upper unitriangular matrices with entries in the field $\F_p$. We shall usually take $\Omega$ to denote an abstract set, and $S^\Omega$ will denote the symmetric group over $\Omega$. Otherwise, if the set being acted upon is irrelevant, we will denote the symmetric group of degree $n$ as $S_n$. We say that a permutation group $G \leq S^\Omega$ is \emph{uniprimitive} if $G$ acts primitively but not doubly transitively on $\Omega$. Generally though, our notation will follow that of Helmut Wielandt in \cite{W}.
	
	\begin{definition}
		Let $H$ be a group. We say that $H$ is a $B$-group if there exists no uniprimitive permutation group containing a regular subgroup isomorphic to $H$.
	\end{definition} 
	The study of $B$-groups was essentially initiated by William Burnside, who stated in the second edition of his book (cf. \cite[\S252]{Burn}) that cyclic groups of prime power order $p^a$ with $a \geq 2$ are $B$-groups. On the other hand, it is easy to see that groups of order $p$ are never $B$-groups, since a subgroup of $S_p$ generated by a $p$-cycle is itself uniprimitive.
	
	Since Burnside's initial work, much work on $B$-groups has been completed by various authors, and adequately summarizing the known results on $B$-groups to date would prove difficult. Using the classification of finite simple groups, Li in \cite[Corollary 1.3]{Li} has proven a complete result for abelian groups of order divisible by at least two distinct primes. Before stating this result, we shall state a definition introduced by Mark Wildon in \cite{Wild2}. 
	
	\begin{definition}
		We say that a group $H$ is $m$\emph{-factorizable} for some integer $m$ if $H$ isomorphic a direct product, say $H \cong \prod_{i=1}^k H_i$ where $k \geq 2$ and $m = |H_1| = |H_2| = \ldots = |H_k|$.
	\end{definition}
	Li's result is as follows:
	\begin{theorem}\label{Li-Corollary}[Li, 2003]
		Let $H$ be an abelian group whose order is divisible by at least two distinct primes. $H$ is not a $B$-group iff $H$ is $m$-factorizable for some integer $m$.
	\end{theorem}

	At least when handling abelian groups, Li's result allows us to restrict our attention to $p$-groups. For the purpose of this note, however, we will instead focus on classifying possibly nonabelian $B$-groups of order $p^4$. The problem of determining whether nonabelian groups belonging to a particular family are $B$-groups has received much less attention, but some results are known (See Theorem \ref{Kochs-thm} below).

	As mentioned above, groups of order $p$ are never $B$-groups. As for groups of order $p^2$, the cyclic groups of this order are $B$-groups by the above mentioned result of Burnside. On the other hand, for odd $p$, elementary abelian $p$-groups of order $p^2$ are never $B$-groups (see Proposition \ref{W25-7} below). When $p=2$ however, one verifies by checking the subgroups of $S_4$ that the Klein-four group is a $B$-group.
	
	The groups of order $p^3$ are somewhat trickier to handle as there are five isomorphism classes to consider. The abelian groups of order $p^3$ were handled by the work of W. Burnside \cite{Burn}, R. Kochend\"orffer \cite{Koch}, and H. Wielandt \cite{W}. The two nonabelian groups of this order were handled by G. Jones in \cite{Jones}. Their results are summarized in the Theorem \ref{BKWJ} below.
	
	For the groups of order $p^4$, we have produced the following result, whose proof we shall give in the last section:
	\begin{thmA}
		If $H$ is a group of order $p^4$ for a prime $p$, then $H$ is a $B$-group iff one of the following holds:
		\begin{itemize}
			\item[(i)] $p$ is any prime, and $H$ is isomorphic to either $C_{p^4}$ or $C_{p^3} \times C_p$.
			\item[(ii)] $p \geq 5$, and $H$ is isomorphic to a group of type (vi), (vii), (xii), or (xiii) in Table 1.
			\item[(iii)] $p = 2$, and $H$ is isomorphic to either $Q_{16}$, $D_{16}$, $Q_8 \times C_2$, or a central product of $D_8$ and $C_4$ given by \texttt{SmallGroup(16,13)}.
			\item[(iv)] $p = 3$, and $H$ is isomorphic to either $C_{27} \rtimes C_3$ (\texttt{SmallGroup(81,6)}), a non-split extension of $C_3^2$ by $C_3^2$ (\texttt{SmallGroup(81,10)}), or a central product of $H_3(3)$ with $C_9$ (\texttt{SmallGroup(81,14)}).
		\end{itemize}  
	\end{thmA}
	
	Our approach is to use the O'Nan-Scott Theorem, which classifies primitive permutation groups into a handful of categories. In our case, the O'Nan-Scott Theorem reduces to the following:
	\begin{theorem}
		If $G \subseteq S_{p^a}$ is primitive for a prime $p$ and an integer $a$, then one of the following holds:
		\begin{itemize}
			\item[(i)] $G$ is an almost simple group acting faithfully on the cosets of a subgroup of index $p^a$. 
			\item[(ii)] $G$ is identified with a subgroup of the primitive wreath product $S_k \wr S_\ell$, where $p^a = k^\ell$.
			\item[(iii)] $G$ is of affine type, i.e., $G$ is a semidirect product of the form $C_p^a \rtimes L$ where $L \leq \GL(a,p)$ acts irreducibly on an $a$-dimensional $\F_p$-vector space.
		\end{itemize}
	\end{theorem}
	
	A few remarks are in order. First, the so-called ``diagonal action'' is omitted since permutation groups of this kind never act on sets of prime power cardinality. Also, perhaps surprisingly, the handling of the almost simple case happens to be the easiest as it turns out that subgroups of prime power index are rather rare in simple groups. On the other hand, handling the product action will prove to be somewhat more delicate, and the majority of our effort is concentrated on analyzing the regular subgroups of the primitive wreath product $S_{p^2} \wr S_2$. As for the affine case, although the primitive groups of affine type possess a unique regular normal subgroup, the group $\AGL(n,p)$ tends to have numerous other regular subgroups. However, these regular subgroups will always have exponent $p$ when $n < p$, and we will have already shown that for $p \geq 5$, the two nonabelian groups of order $p^4$ and exponent $p$ are non-$B$-groups through handling the product action. Thus, our handling of the affine case will be relatively short.
	
	In the course of our analysis, we will be able to prove the following result, which serves as a weaker version of Li's result for $p$-groups.
	\begin{thmB}
		Let $H$ be a group of order $p^q$ for primes $p$ and $q$ where $q < p$ and $\exp(H) \geq p^2$. Then, $H$ is a $B$-group. 
	\end{thmB}
	Notice that our assumptions above imply that $H$ is not $m$-factorizable for any $m \geq 2$. The assumption that $q < p$ cannot be relaxed either. For instance, if $p = q = 3$, then the nonabelian group of order $27$ and exponent $9$ appears as a regular subgroup of $\PSU(4,2)$ in its uniprimitive action on 27 points. 
	
	This work is organized as follows. In section 2, we will review some of the known results on $B$-groups that will prove useful to us. In sections 3, 4, and 5, we will respectively handle the almost simple case, product action case, and affine case. As many of our arguments only apply for primes greater than 3, we separately handle the primes 2 and 3 in section 6. Finally, in the last section, we will use the results collected in the previous sections to prove Theorems A and B.

	\section{Preliminaries}
	
	We establish some important results which are either important for historical context or will prove useful in the foregoing arguments. 
	
	\subsection{Some classical results on $B$-groups}
	
	We first state Burnside's result from \cite[\S252]{Burn}:
	
	\begin{theorem}\label{Burnside-cyclic}
		Let $G$ be a permutation group acting uniprimitively on a set $\Omega$ with $|\Omega| = p^a$ for some prime $p$ and an integer $a$. Further, assume that $G$ possesses a cyclic, regular subgroup $H$. Then, $|H| = p$.
	\end{theorem}
	In particular, this implies that cyclic groups of order $p^a$ with $a > 1$ are $B$-groups. Interestingly, Burnside's argument in \cite[\S252]{Burn} contains a significant error. However, there has since appeared numerous alternative proofs of this result. See \cite{Knapp} or \cite{Wild2} for a more detailed discussion.
	
	On the other hand, if a primitive permutation group contains a regular subgroup of order $p$, Burnside provides an elegant character-theoretic proof for the fact that any such group is necessarily solvable in \cite[\S251]{Burn}.
	
	Although Theorem \ref{Li-Corollary} relies upon the classification of finite simple groups, the ``if'' direction can be seen in a straightforward manner by considering the product action of wreath products. In fact, this claim remains true under a much milder assumption on the structure of $G$. The proof provided is essentially that of Wielandt's Theorem 25.7 in \cite{W}.
	
	\begin{proposition}\label{W25-7}
		Let $m \geq 2$ be an integer, and let $H$ be an $m$-factorizable group that is not isomorphic to an elementary abelian 2-group of prime rank. Then, $H$ is not a $B$-group. 
	\end{proposition}
	\begin{proof}
		Write $H$ as a direct product, say $H = \prod_{i=1}^k H_i$ for $k \geq 2$, and assume that each $|H_i|$ has equal order, say $m$. We may assume that a decomposition can be chosen so that $m \geq 3$ except in the case where $H$ is an elementary abelian 2-group of prime rank. 
		
		Let $\Omega$ be a set with cardinality $m$. Each $H_i$ can be embedded as a regular subgroup of the symmetric group $S^\Omega$, and we can thereby embed $H$ as a regular subgroup of the $k$-fold direct product of $S^\Omega$ which acts naturally on the $k$-fold direct product $\Omega^k$. In particular, $H$ is a regular subgroup of the wreath product $S^\Omega \wr S_k$. This action is guaranteed to be primitive if $m \geq 3$ and $k \geq 2$ (see \cite[Lemma 2.7A]{DM} for instance). On the other hand, $S^\Omega \wr S_k$ in its product action is easily seen to not be doubly transitive.
	\end{proof}
	It is interesting to note that there is yet no complete criterion for an elementary abelian 2-group to be a $B$-group, and indeed, the answer depends upon the rank. It is known that an elementary abelian 2-group of rank $a$ is not a $B$-group whenever $2^a - 1$ is composite (cf. \cite{Wild2}).
	
	We now state some miscellaneous classical results on $B$-groups.
	\begin{theorem} \label{Kochs-thm}
		\hspace{1cm}
	\begin{itemize}
		\item[(i)] [Kochend\"orffer, 1937]	Let $p$ be a prime, and let $a$ and $b$ be distinct positive integers. The group $C_{p^a} \times C_{p^b}$ is a $B$-group. \cite{Koch}
		\item[(ii)] [Wielandt, 1950] Let $n \geq 3$ be an integer. The dihedral group $D_{2n}$ is a $B$-group. \cite{W2}
		\item[(iii)] [Scott, 1957] The generalized dihedral groups are $B$-groups. \cite{Scott}
	\end{itemize}
	\end{theorem}
	
	Notice that groups of the form $C_{p^a} \times C_{p^b}$ with $a \neq b$ are not uniformly decomposable. Thus, Kochend\"orffer's result leads us to wonder whether Theorem \ref{Li-Corollary} at least somewhat generalizes to abelian $p$-groups where $p$ is odd. 
	
	\begin{question}
		Let $H$ be a nonsimple abelian $p$-group for a prime $p$, and assume that $H$ is not uniformly decomposable. Then, is $H$ a $B$-group?
	\end{question}

\subsection{Groups of order $p^3$}

The classification of $B$-groups of order $p^3$ was essentially completed by G. Jones in \cite{Jones} with the aid of the above-mentioned and highly nontrivial results of Burnside, Kochend\"orffer, and Wielandt. 

\begin{theorem}\label{BKWJ}
	Let $H$ be a group of order $p^3$ for some prime $p$. Then $H$ is a $B$-group iff one of the following holds:
	\begin{itemize}
		\item[(i)] $p = 2$.
		\item[(ii)] $p \geq 5$, and $H$ is isomorphic to $C_9 \rtimes C_3$.
		\item[(iii)] $p$ is any prime, and $H$ is isomorphic to $C_{p^3}$ or $C_{p^2} \times C_p$.
	\end{itemize}
\end{theorem}
A few remarks are in order. First, the fact that no group of order 8 is a $B$-group is in particular a consequence of that fact that there are no uniprimitive groups acting with degree 8 (see Table 1 in \cite{Wild2} for instance). Also, the two nonabelian groups of order 27 are seen to be regular subgroups of $\PSU(4,2)$ in its uniprimitive action of degree 27 and thus are not $B$-groups (see Lemma 3.3 below). We lastly remark that the group $H_3(p)$ for $p \geq 5$ is not a $B$-group as it is isomorphic to a regular subgroup of the unirpimitive group $C_p^3 \rtimes \SO(3,p)$ acting on $C_p^3$ viewed as a $\F_p$-vector space. Indeed, consider the following subgroup of $\AGL(3,p)$:
\begin{equation}
	\bigg\la \textbf{x}
	\begin{bmatrix} 
		1 & x & \frac{1}{2}x^2 \\ 
		0 & 1 & x \\
		0 & 0 & 1
	\end{bmatrix} +
	\begin{bmatrix}
		x \\ y \\ z
	\end{bmatrix}^T \bigg| x,y,z \in \F_p \bigg\ra.
\end{equation}
It is routine to verify that this group is isomorphic to $H_3(p)$ and acts regularly on the $\F_p$-vector space of dimension $3$. The claim then follows from constructing a nondegenerate bilinear form that is preserved by the linear component of each group element.
	
	We now use the above results to tabulate the abelian $B$-groups of order $p^4$. We state this as a lemma for future use.
	\begin{lemma}
		Let $H$ be an abelian of order $p^4$ for a prime $p$. $|H|$ is a $B$-group iff $H$ is isomorphic to $C_{p^4}$ or $C_{p^3} \times C_p$.
	\end{lemma}
	\begin{proof}
		The fact that cyclic groups of composite prime power order are $B$-groups follows from Theorem \ref{Burnside-cyclic}. The fact that the group $C_{p^3} \times C_{p}$ is a $B$-group follows from Theorem \ref{Kochs-thm}(i). The remaining three isomorphism classes of abelian groups are seen not to be $B$-groups by Proposition \ref{W25-7}.
	\end{proof}
	
	\subsection{Nonabelian groups of order $p^4$}
	
	Now that the abelian groups have been accounted for, we provide in Table 1 presentations for the nonabelian groups of order $p^4$ for $p \geq 5$ as described by Burnside in \cite[\S117]{Burn}. We also provide the exponent and the isomorphism classes of the center and derived subgroup for each group as these will prove useful in our analysis of the product action.

	\begin{table}[htpb]
		\centering
		\caption{Nonabelian groups of order $p^4$ for $p \ge 5$ \cite[\S117]{Burn}. For the provided information about the isomorphism classes of the centers and derived subgroups, see \cite{HA}.}
		\label{tab:burnside_groups}
		\renewcommand{\arraystretch}{1.4}
		\begin{tabular}{@{}| l | p{10.5cm} | c | c | c |@{}}
			\hline
			\textbf{Type} & \textbf{Presentation} & $G'$ & $\Cent(G)$ & $\exp(G)$ \\
			\hline
			\textbf{(vi)}& 
			$\la a,b \mid a^{p^3} = b^p = 1, a^b = a^{p^2 + 1} \ra$ & 
			$C_p$ & $C_{p^2}$ & $p^3$ \\
			\hline
			
			\textbf{(vii)} & 
			$\la a,b,c \mid  a^{p^2} = b^p = c^p = [a,b] = [a,c] = 1, b^c = ba^p\ra$ & 
			$C_p$ & $C_{p^2}$ & $p^2$ \\
			\hline
			
			\textbf{(viii)} & 
			$
			\la a,b\mid a^{p^2} = b^{p^2} = 1, a^b = a^{p + 1}\ra
			$ & 
			$C_p$ & $C_p^2$ & $p^2$ \\
			\hline
			
			\textbf{(ix)} & 
			$\la a,b,c \mid a^{p^2} = b^p = c^p = [a,b] = [b,c] = 1, a^c = a^{p + 1}\ra$ & 
			%\textit{(Direct product of $\langle Q \rangle$ and $\langle P, R \rangle$)}  
			$C_p$ & $C_p^2$ & $p^2$ \\
			\hline
			
			\textbf{(x)} & 
			$\la a,b,c \mid a^{p^2} = b^p = c^p = [a,b] = [b,c] = 1, a^c = ab \ra$ & 
			$C_p$ & $C_p^2$ & $p^2$ \\
			\hline
			
			\makecell{\textbf{(xi)--(xiii)}} & \begin{minipage}[t]{10.5cm}$\la a,b,c \mid a^{p^2} = b^p = c^p = 1, a^b = a^{p + 1}, a^c = ab, b^c = a^{\gamma p}b \ra$ \newline 
			\textit{where $\gamma = 0$ for (xi), $\gamma = 1$ for (xii),} \newline 
			\textit{and $\gamma$ is any non-residue modulo $p$ for (xiii).} \end{minipage} & 
			$C_p^2$ & $C_p$ & $p^2$ \\
			\hline
			
			\textbf{(xiv)} & 
			$\la a,b,c,d \mid a^p = b^p = c^p = d^p = [\la a,b,c,d\ra, \la a,b \ra] =  1, c^d = ca \ra$ & 
			$C_p$ & $C_p^2$ & $p$ \\
			\hline
			
			\textbf{(xv)} & 
			$\la a,b,c,d \mid a^p = b^p = c^p = d^p = [a,\la b,c,d \ra] = [b,c] = 1, b^d = ba, c^d = cb \ra$  & 
			$C_p^2$ & $C_p$ & $p$ \\
			\hline
		\end{tabular}
	\end{table}

\section{Almost Simple Groups}

Let $G$ be a group acting transitively on a set $\Omega$, and choose $\alpha \in \Omega$. Recall that the \emph{rank} of $G$ is the number of orbits of $G_\alpha$ in its action on $\Omega$ including the orbit $\{\alpha\}$. We have the following result of R. Guralnick in \cite[Corollary 2]{Gur}.
\begin{theorem}
	Let $G$ be a finite simple group having a subgroup $H$ of index $p^a$ where $p$ is prime and $a$ is an integer. The action of $G$ on the cosets of $H$ is doubly transitive except in the case where $G \cong \PSU(4,2)$ and $|G:H| = 27$, in which case, the action is uniprimitive of rank 3. 
\end{theorem}

In particular, the only simple group that can act uniprimitively on any set of prime power cardinality is $\PSU(4,2)$. Using this result, we are able to prove the following result for almost simple groups:

\begin{proposition}
	Let $G$ be an almost simple group which acts faithfully and uniprimitively on a set $\Omega$ where $|\Omega| = p^a$ for some prime $p$ and an integer $a$. If $\alpha \in \Omega$, then $|G:G_\alpha| = 27$, and $\Soc(G) \cong \PSU(4,2)$.
\end{proposition}	
\begin{proof}
	Let $G$ be an almost simple group, and let $G$ act faithfully and primitively on a set $\Omega$ where $|\Omega| = p^a$. Also, set $S = \Soc(G)$. Clearly, we may assume that $G$ is nonsimple, for otherwise, the statement reduces to the above result of Guralnick. Choose $\alpha \in \Omega$. Since $G$ acts faithfully, $S$ is not a subgroup of $G_\alpha$. Conversely, $G_\alpha$ cannot be a proper subgroup of $S$, since $G_\alpha$ is necessarily maximal by primitivity. Consequently, $SG_\alpha = G$. By the second isomorphism theorem, we have that \newline $|S:S\cap G_\alpha| = p^a$. 
	
	By Guralnick's result above, the action of $S$ on the cosets of $S \cap G_\alpha$ is doubly transitive unless $S \cong \PSU(4,2)$ and $p^a = 27$. Under the latter circumstance, there is nothing to prove, so we can assume that $S$ acts doubly transitively on the cosets of $S \cap G_\alpha$. It now suffices to show that the action of $G$ on the cosets of $G_\alpha/G$ is doubly transitive.
	
	Consider the map $(S \cap G_\alpha)s \longmapsto G_\alpha s$ from $(S \cap G_\alpha)/S$ to $G_\alpha/G$. It is straightforward to verify that this map is bijective and does not depend on a choice of coset representatives. In particular, a set of coset representatives for $S$ over $S \cap G_\alpha$ is a set of coset representatives for $G$ over $G_\alpha$, and for $s_1, s_2 \in S$, we have $G_\alpha s_1 = G_\alpha s_2$ iff $(S \cap G_\alpha)s_1 = (S \cap G_\alpha)s_2$.  Now, given two pairs of cosets, say $(G_\alpha s_1, G_\alpha s_2)$ and $(G_\alpha s_3, G_\alpha s_4)$ for $s_i \in S$ and $i=1, 2, 3, 4$, we want to find $g \in G$ such that 
	$$
		(G_\alpha s_1 g, G_\alpha s_2 g) = (G_\alpha s_3, G_\alpha s_4).
	$$
	By the double transitivity of $S$, we can find $s \in S$ such that
	$$
		((S \cap G_\alpha) s_1 s, (S \cap G_\alpha) s_2 s) = ((S \cap G_\alpha) s_3, (S \cap G_\alpha) s_4).
	$$
	From this, we obtain
	$$
		(G_\alpha s_1 s, G_\alpha s_2 s) = (G_\alpha s_3, G_\alpha s_4),
	$$
	and it follows that $G$ acts doubly transitively on the cosets of $G_\alpha$ and therefore also acts doubly transitively on $\Omega$.
\end{proof}
	
Thus, the only uniprimitive permutation group of almost simple type is given by $\PSU(4,2)$ and its action on the cosets of a subgroup of index 27. In particular, when searching for $B$-groups of order $p^a$ for $a \geq 4$, we need not consider the almost simple case whatsoever.

On the other hand, $\PSU(4,2)$ in its uniprimitive action of 27 points indeed possesses regular subgroups, including both of the two nonabelian groups of order 27, as is easily verified using GAP \cite{Gap}. We state this as a lemma.

\begin{lemma}\label{reg-of-PSU42}
	Let $\PSU(4,2)$ act uniprimitively on a set $\Omega$ with $|\Omega| = 27$. Under this action, $\PSU(4,2)$ contains regular subgroups isomorphic to both the semidirect product $C_9 \rtimes C_3$ and the Heisenberg group $H_3(3)$.
\end{lemma}

One approach to verify this lemma is to construct the permutation character in GAP and to search for the subgroups of order 27 for which this character restricts to the regular character. Since the action of $\PSU(4,2)$ is of rank 3, the permutation character has three irreducible constituents, each distinct from one another (see \cite[Corollary 5.16]{Is} for instance). One then verifies from inspecting the character table of $\PSU(4,2)$ that the permutation character is the sum of the principal character and the unique irreducible characters of degrees 6 and 20.

\section{The Product Action}

For the purpose of this section, any wreath product $S_k \wr S_\ell$ is assumed to be acting with the product action on $k^\ell$ points. Let $G$ be a primitive subgroup of $S_k \wr S_\ell$. Any regular subgroup of $G$ is also a regular subgroup of $S_k \wr S_\ell$. Since $S_k \wr S_\ell$ is never doubly transitive (except in the trivial cases where either $k$ or $\ell$ equal 1, which are subsumed by the almost simple case), it suffices only to consider the regular subgroups of $S_k \wr S_\ell$ for the purposes of classifying $B$-groups.

\begin{lemma}
	Let $p$ be a prime and let $a$ be an integer where $a < p$. The regular subgroups of the wreath product $S_p \wr S_a$ in its product action on $p^a$ points are precisely the Sylow $p$-subgroups of $S_p \wr S_a$ and are elementary abelian.
\end{lemma}
\begin{proof}
	Clearly, the Sylow $p$-subgroups of $\prod_{i=1}^a S_p$ are elementary abelian and act regularly under the product action. Since $a < p$, we have that $p \nmid a!$. Thus, the Sylow $p$-subgroups of this direct product will remain the Sylow $p$-subgroups of $S_p \wr S_a$.
\end{proof}

This essentially handles the product action case when handling groups of order $p^q$ for primes $p$ and $q$ where $q < p$. Indeed, the only nontrivial primitive wreath product of degree $p^q$ to consider is $S_p \wr S_q$. 

In handling the $p$-groups of order $p^4$, there are essentially two wreath products to consider, namely $S_p \wr S_4$ and $S_{p^2} \wr S_2$. In the former case and when $p \geq 5$, the regular subgroups are again all elementary abelian $p$-groups by the above lemma. On the other hand, investigating the regular subgroups of $S_{p^2} \wr S_2$ is somewhat more delicate since the Sylow $p$-subgroups are much larger. The lemma below will help us handle the latter situation.

\begin{lemma}\label{fund_prod_lem}
	Let $p \neq 2$ be a prime, and let $H$ be a group of order $p^4$. $H$ is isomorphic to a regular subgroup of $S_{p^2} \wr S_{2}$ in its product action on $p^4$ points iff $H$ possess a pair of subgroups $K_1$ and $K_2$ satisfying the following properties:
	\begin{itemize}
		\item[(a)] $|K_1| = |K_2| = p^2$.
		\item[(b)] $K_1^{g_1} \cap K_2^{g_2} = 1$ for each $g_1,g_2 \in H$.
	\end{itemize}  
\end{lemma}
	
\begin{proof}
	Let $\Omega$ be a set with $|\Omega| = p^2$. Since we are assuming that $p \neq 2$, any Sylow $p$-subgroup of $S_{\Omega} \wr S_2$ is contained in the base group $S_{\Omega} \times S_{\Omega}$. Thus, we only need to consider regular subgroups of the base group. 
	
	First, assume that $H$ is a regular subgroup of $S^{\Omega} \times S^{\Omega}$. Choose $\alpha \in \Omega$, and let $H_{(\alpha, \cdot)}$ denote the subgroup of $H$ that permutes the set $\{\alpha\} \times \Omega$. In particular, $H_{(\alpha,\cdot)}$ consists of the elements of $H$ which map $(\alpha,\gamma)$ to $(\alpha,\delta)$ for $\gamma,\delta \in \Omega$. Likewise, for $\beta \in \Omega$, define $H_{(\cdot, \beta)}$ to be the subgroup of $H$ that permutes $\Omega \times \{\beta\}$. Set $K_1 := H_{(\alpha,\cdot)}$ and $K_2 := H_{(\cdot,\beta)}$. We claim that this choice of $K_1$ and $K_2$ satisfies the two axioms above.
	
	The fact that $K_1$ and $K_2$ have order $p^2$ will follow from the Orbit-Stabilizer Theorem. Of course, by symmetry, it suffices only to verify the claim for $K_1$. Consider the action of $S^\Omega \times S^\Omega$ on $\Omega$ given by setting $\alpha^{(g,h)} := \alpha^g$. Since $H$ is transitive on $\Omega \times \Omega$, it is also transitive (though not necessarily faithful) under this action on $\Omega$. Also, observe that the stabilizer of $\alpha$ under this new action is $H_{(\alpha,\cdot)}$. Thus, by the Orbit-Stabilizer Theorem, $ H_{(\alpha,\cdot)}$ has order $p^2$, which shows that (a) holds.
	
	Now, we show that (b) holds for our choice of $K_1$ and $K_2$. One sees that any conjugate of $H_{(\alpha,\cdot)}$ in $H$ will be of the form $H_{(\gamma,\cdot)}$ for some $\gamma \in \Omega$. The analogous statement holds as well for $H_{(\cdot,\beta)}$. Thus, in searching for a contradiction, it is of no loss to assume that $H_{(\alpha,\cdot)} \cap H_{(\cdot,\beta)} \neq 1$. If there exists a nonidentity element $(g,h)$ lying in this intersection, then $(\alpha,\beta)^{(g,h)} = (\alpha,\beta)$, and so $(g,h)$ fixes $(\alpha,\beta)$. However, as $H$ is assumed to be regular, this implies $(g,h)$ is the identity. Thus, (b) holds for our choice of $K_1$ and $K_2$. 
	
	Now, we prove the converse. Assume $H$ possesses two subgroups $K_1$ and $K_2$ satisfying the above two properties. Let $\Omega_1$ and $\Omega_2$ denote the set of cosets of $K_1$ and $K_2$ respectively in $H$. Clearly, property (a) implies that $|\Omega_1| = |\Omega_2| = p^2$. The action of $H$ on these sets of cosets gives us a pair of homomorphisms $\rho_1: H \to S^{\Omega_1}$ and $\rho_2: H \to S^{\Omega_2}$. In turn, we get a homomorphism $\rho_1 \times \rho_2: H \to S^{\Omega_1} \times S^{\Omega_2}$ given by $(\rho_1 \times \rho_2)(h) = (\rho_1(h),\rho_2(h))$. It follows from property (b) that $\rho_1 \times \rho_2$ is injective, for otherwise, there would exist a group element $h \in H$ leaving all cosets of both $K_1$ and $K_2$ fixed, which would imply $h \in K_1 \cap K_2$. Likewise, property (b) also implies that no nonidentity element of $\bar{H}$ stabilizes any point of $\Omega_1 \times \Omega_2$, for otherwise, this would imply that there exists conjugates of $K_1$ and $K_2$ that intersect nontrivially. The fact that $\bar{H}$ acts regularly on $\Omega_1 \times \Omega_2$ now follows from the Orbit-Stabilizer Theorem and the fact that $|\bar{H}| = |H| = |\Omega_1 \times \Omega_2| = p^4$. This proves the lemma.
\end{proof}
	
The utility of this lemma is that we only need to consider the purely group-theoretic properties of the groups in Table 1.

Before going on to prove the main result of this section, we state another lemma without proof which will be of use to us during our analysis. This lemma appears as Corollary 12.3.1 in \cite{Hall}.

\begin{lemma}\label{hall-pet}
	Let $H$ be any $p$-group for a prime $p$, and assume that the nilpotency class of $H$ is less than $p$. Then, for any $g,h \in H$ and any integer $a$, we have that
	$$
	(gh)^{p^a} = g^{p^a}h^{p^a}x^{p^a}y^{p^a}
	$$
	for some $x,y \in H'$. In particular, if $H$ order $p^4$ and $p \geq 5$, then since $H'$ has exponent at most $p$ (see Table 1), it follows that $(gh)^p = g^ph^p$.
\end{lemma}

We now prove the main result of this section:

\begin{theorem}\label{main_thm_prod}
	Let $H$ be a nonabelian group of order $p^4$ for a prime $p \geq 5$. Then $H$ appears as a regular subgroup of $S_{p^2} \wr S_2$ in its product action on $p^4$ points unless $H$ is isomorphic to a group of type (vi), (vii), (xii), or (xiii) in Table 1.
\end{theorem}
\begin{proof}
	As in the proof of Lemma 4.2, since we assume $p \neq 2$, any regular subgroup of $S_{p^2} \wr S_2$ is a regular subgroup of the base group $S_{p^2} \times S_{p^2}$, so it suffices to only search for regular subgroups of the base group.
	
	We first show that the groups of type (vi) and (vii) cannot be embedded as regular subgroups in $S_{p^2} \times S_{p^2}$, so assume $H$ is isomorphic either one of these groups. By Lemma \ref{fund_prod_lem}, it suffices to show that we cannot produce a pair of trivially intersecting subgroups $K_1$ and $K_2$ both having order $p^2$. Assume for a contradiction that we are able to choose such a pair of groups. If $H$ is either of type (vi) or (vii), then $\Cent(H) \cong C_{p^2}$. Thus, $K_1$ and $K_2$ cannot both intersect nontrivially with the center. However, if $K_1 \cap \Cent(H) = 1$ for instance, then $\Cent(H)K_1$ is abelian of order $p^4$, an impossibility.
	
	We now handle the groups of types (xii) and (xiii) simultaneously, both of which have exponent $p^2$. We claim that all subgroups of order $p^2$ intersect with the center, which in either case is seen to be $\la a^p \ra$. Setting $H$ to be either one of these groups, it follows from Lemma 4.3 that the subgroup $\Omega_1(H)$ of $H$ generated by all elements of order $p$ is isomorphic to $H_3(p)$, which has exponent $p$. Also, any subgroup of $H_3(p)$ isomorphic to $C_p \times C_p$ intersects with $\Cent(H_3(p))$ as is easily verified. Seeing that $\Cent(\Omega_1(H)) = \Cent(H)$, we conclude that any subgroup of $H$ isomorphic to $C_p \times C_p$ is contained in $\Omega_1(H)$ and therefore contains $\Cent(H)$. On the other hand, any cyclic subgroup of order $p^2$ contains the center $\la a^p \ra$. Indeed, assume $o(a^ib^jc^k) = p^2$ for integers $i$, $j$, and $k$. By Lemma \ref{hall-pet}, we have $(a^ib^jc^k)^p = a^{ip}(b^jc^k)^p = a^{ip}$. It follows now that every subgroup of order $p^2$ contains $\Cent(H)$, and consequently no suitable $K_1$ and $K_2$ can be chosen satisfying the conditions of Lemma \ref{fund_prod_lem}.
	
	For each of the remaining types of groups in Table 1, we proceed by constructing subgroups $K_1$ and $K_2$ which satisfy properties (a) and (b) in Lemma \ref{fund_prod_lem}.
	
	For the group of type (viii) in Table 1, we can simply choose $K_1 = \la a \ra$ and $K_2 = \la b \ra$. Both groups are cyclic of order $p^2$. To see that property (b) of Lemma \ref{fund_prod_lem} holds, observe that $\la a^p \ra$ and $\la b^p \ra$ are distinct subgroups of the group's center. 
	
	We can handle the groups of types (ix), (x) and (xi) simultaneously. Each of these groups possess generators $a$, $b$, and $c$ where $o(a) = p^2$ and $\la b,c \ra \cong C_p \times C_p$. Set $H$ to be either one of these groups. We claim that setting $K_1 = \la a \ra$ and $K_2 = \la b,c \ra$ will satisfy the properties in Lemma \ref{fund_prod_lem}. As above, it follows from Lemma \ref{hall-pet} that any cyclic subgroup of order $p^2$ in $H$ contains the central subgroup $\la a^p \ra$. From this, it follows that no conjugate of $\la b,c \ra$ intersects with $\la a \ra$. Otherwise, $a^p$ would be conjugate to an element of $\la b,c \ra$, and $a$ would therefore be conjugate to an element of order $p^2$ whose $p$th power lies in $\la b,c \ra$, an impossibility.
	
	Setting $H$ to be the group of type (xiv) in Table 1, we set $K_1 = \la a,c \ra$ and $K_2 = \la b,d \ra$. Both groups are isomorphic to $C_p \times C_p$. One observes that $\la a,c \ra$ is normal. On the other hand, one calculates that the conjugates of $\la b,d \ra$ are of the form $\la b, a^id \ra$ for $i \in \Z$. Now, $\la a,c \ra \cap \la b, a^id \ra = 1$ for each $i$, so the two chosen subgroups satisfy the conditions in Lemma \ref{fund_prod_lem}.
	
	Lastly, setting $H$ to be the group of type (xv) in Table 1, we set $K_1 = \la a,d \ra$ and $K_2 = \la b,c \ra$. Again, both groups are isomorphic to $C_p \times C_p$. Observe that $\la a,b,c \ra \cong C_p^3$ and that $H \cong \la a,b,c \ra \rtimes d$. In particular, $\la a,b,c \ra \cap \la a,d \ra^g = \la a \ra = \Cent(H)$ for any $g \in H$.  Now, any conjugate of $\la b,c \ra$ lies in $\la a,b,c \ra$ since $\la a,b,c \ra \triangleleft H$. On the other hand, $\la b,c \ra \cap \Cent(H) = 1$, so no conjugate of $\la b,c \ra$ will intersect with the center, and it follows that no conjugate of $\la a,d \ra$ will intersect with any conjugate of $\la b,c \ra$. Thus, our choice of $K_1$ and $K_2$ satisfy the conditions in Lemma \ref{fund_prod_lem}.
\end{proof}

\section{Groups of Affine Type}
	
We now handle the groups of affine type. It will follow from the below proposition that the Sylow $p$-subgroups of $\AGL(4,p)$ will have exponent $p$, and so, any regular subgroup of $\AGL(4,p)$ under its action on the vector space $\F_p^4$ will have exponent $p$ as well. However, we have already established in the previous section that the two nonabelian exponent-$p$ groups of order $p^4$ are not $B$-groups since they appear as regular subgroups of the primitive wreath product $S_{p^2} \wr S_2$. Thus, for primes at least 5, no new groups are excluded from being $B$-groups from considering the affine case.

\begin{proposition}
	If $p$ is a prime and $n < p$, then a Sylow $p$-subgroup of $\AGL(n,p)$ has exponent $p$. In particular, any regular subgroup of $\AGL(n,p)$ in its action on the vector space $\F_p^n$ for $n < p$ will have exponent $p$. 
\end{proposition}
\begin{proof}
	We can express $\AGL(n,p)$ as $C_p^n \rtimes \GL(n,p)$. We can choose a Sylow $p$-subgroup of $\GL(n,p)$ to be the subgroup consisting of upper unitriangular matrices, which we denote as $H_n(p)$. The group $H_n(p)$ is well known to have exponent $p$ when $n \leq p$.
	
	Now, a Sylow $p$-subgroup of $\AGL(n,p)$ is isomorphic to $C_p^4 \rtimes H_n(p)$. We take elements of this group to be of the form $xA + v$, where $x$ and $v$ are row vectors of length $n$ with entries in $\F_p$ and $A \in H_n(p)$. One now sees that the $p$-fold composition of this affine transformation is given by
	$$
		xA^p + v(A^{p-1} + A^{p-2} + \ldots + I_n).
	$$
	As $H_n(p)$ has exponent $p$, the matrix $A^p$ is the identity, so it suffices to show that $\sum_{i=0}^{p-1} A^i = 0$. 
	
	We write $A = I_n + N$ where the entries of $N$ are zero on and below the main diagonal. We have
	$$
		\sum_{i=0}^{p-1} A^i = \sum_{i=0}^{p-1} (I_n + N)^i = \sum_{i=0}^{p-1} \sum_{j=0}^i \binom{i}{j} N^j = \sum_{j=0}^{p-1} N^j \sum_{i=j}^{p-1} \binom{i}{j}.
	$$ 
	Using Fermat's Identity, we have $\sum_{i=j}^{p-1} \binom{i}{j} = \binom{p}{j+1}$, so the above sum is equal to $$
		\sum_{j=0}^{p-1} \binom{p}{j+1} N^j = \binom{p}{1}I_n + \binom{p}{2}N + \ldots + \binom{p}{p-1}N^{p-2} + \binom{p}{p}N^{p-1}.
	$$
	Now, all of these binomial coefficients except for the last vanish in characteristic $p$. However, the last term vanishes as well since $N^n = 0$ and we assume $n < p$. Thus, the sum is zero, as desired.
\end{proof}

We remark that our assumption above that $n < p$ is necessary. For instance, the Sylow 3-subgroups of $\AGL(4,3)$ have exponent 9.

\section{Groups of order 16 and 81}
	
The groups of orders 16 and 81 are small enough to handle with computer calculations performed in GAP \cite{Gap}. For simplicity, rather than creating a table as we have in Table 1, we instead use the GAP identifiers to refer to each group when necessary. We list the $B$-groups for each of the two orders under consideration in their own lemmas. The lemmas were verified using GAP's Primitive Groups library, and the code used to verify these results is given in Listings 1 and 2 in the appendix.
	
\begin{lemma}\label{order16}
	If $H$ has order 16, then $H$ is a $B$-group iff $H$ is isomorphic to one of the following groups:
	\begin{itemize}
		\item[(i)] $C_{16}$,
		\item[(ii)] $C_{8} \times C_2$,
		\item[(iii)] $Q_{16}$,
		\item[(iv)] $D_{16}$,
		\item[(v)] $Q_8 \times C_2$, or
		\item[(vi)] A central product of the dihedral group $D_8$ with $C_4$ (\texttt{SmallGroup(16,13)}).
	\end{itemize} 
\end{lemma}

\begin{lemma}\label{order81}
If $H$ has order 81, then $H$ is a $B$-group iff $H$ is isomorphic to one of the following groups:
\begin{itemize}
	\item[(i)] $C_{81}$,
	\item[(ii)] $C_{27} \times C_3$,
	\item[(iii)] $C_{27} \rtimes C_3$ (\texttt{SmallGroup(81,6)}),
	\item[(iv)] A non-split extension of $C_3^2$ by $C_3^2$ (\texttt{SmallGroup(81,10)}), or
	\item[(v)] A central product of the $H_3(3)$ with $C_9$ (\texttt{SmallGroup(81,14)}).
\end{itemize} 
\end{lemma}

\section{Proof of Theorems A and B}

We have now collected enough results to give proofs for our two main theorems.

\begin{proof}[Proof of Theorem A]
	Let $H$ be a regular subgroup of a uniprimitive permutation group $G$ acting with degree $p^4$. In the case where $H$ is abelian, Theorem A holds by Lemma 2.6. If $p = 2$ or $p = 3$, Theorem A holds by Lemmas \ref{order16} and \ref{order81} respectively. Thus, we may assume that $H$ is nonabelian and that $p \geq 5$ for the remainder of the proof.
	
	By the O'Nan-Scott Theorem, there are essentially three cases to conider: either $G$ is permutation isomorphic to an almost simple group acting on the cosets of a subgroup of index $p^4$, a subgroup of either of the primitive wreath products $S_{p} \wr S_4$ or $S_{p^2} \wr S_2$, or a subgroup of the semidirect product of $C_p^4 \rtimes \GL(4,p)$ acting on the $\F_p$ vector space of dimension 4. By Proposition 3.2, there are no uniprimitive almost simple groups acting with degree $p^4$. Further, in considering the product action, since $S_{p} \wr S_4$ and $S_{p^2} \wr S_2$ are both uniprimitive, we may assume without loss that $G$ is isomorphic to either of these wreath products in the product action case. 
	
	In the product action case, we have shown in Lemma 4.1 that the only regular subgroups of $S_p \wr S_4$ are the elementary abelian $p$ groups of rank 4. Now, the four isomorphism types listed in Theorem \ref{main_thm_prod} which do not appear as regular subgroups of $S_{p^2} \wr S_2$ are precisely the groups listed in part (iii) of Theorem A. It now suffices to determine that the affine group $\AGL(4,p)$ does not contain any of these four groups as regular subgroups. However, by Proposition 5.1, it follows that any regular subgroup of $\AGL(4,p)$ has exponent $p$, but these groups were already shown to not be $B$-groups in Theorem \ref{main_thm_prod}. 
\end{proof}

\begin{proof}[Proof of Theorem B]
	Let $H$ be a regular subgroup of a uniprimitive permutation group $G$ acting with degree $p^q$ for primes $p$ and $q$ where $q < p$. By the O'Nan-Scott Theorem, there are essentially three cases to conider: either $G$ is permutation isomorphic to an almost simple group acting on the cosets of a subgroup of index $p^q$, a subgroup of the primitive wreath product $S_{p} \wr S_q$, or a semidirect product of $C_p^q$ being acted upon by an irreducible subgroup of $\GL(q,p)$. Again, by Proposition 3.2, there is no such uniprimitive group $G$ of almost simple type. By Lemma 4.1, any regular subgroup of $S_p \wr S_q$ in its product action on $p^q$ points will be elementary abelian and hence will have exponent $p$. In case where $G$ has affine type, Proposition 5.1 implies that $H$ will have exponent $p$. Thus, in any case where such an $H$ can exist, $H$ will have exponent $p$. Theorem B now follows.
\end{proof} 

A natural next step would be to attempt a classification of $B$-groups of order $p^5$ where $p$ is any prime. The groups of this order have been classified (cf. \cite{Bender}). If $p \geq 7$, then the O'Nan-Scott Theorem, Proposition 3.2, and Lemma 4.1 imply that we need only search among the uniprimitive subgroups of $\AGL(5,p)$, whose Sylow $p$-subgroups have exponent $p$. Interestingly, in the case $p = 2$, there are no uniprimitive groups acting with degree 32, so all groups of this order are $B$-groups. On the other hand, for the primes 3 and 5, the $B$-groups can be found using the GAP script in Listing 1. Indeed, the Small Groups library contains all groups of orders 243 and 3,125, and the Primitive Groups library contains all groups acting primitively with degree less than 4,096. Thus, a complete classification of the $B$-groups of order $p^4$ amounts to answering the following question:

\begin{question}
	Which regular subgroups of $\AGL(5,p)$ lie in a uniprimitive subgroup of $\AGL(5,p)$?
\end{question}

\printbibliography[heading=bibintoc]

\newpage
\section{Appendix}
	
\begin{lstlisting}[caption={This is a GAP script which searches among the uniprimitive groups of degree ``prime\_power'' for non-B-groups of prime power order. A subgroup is tested for regularity by determining whether the permutation character of the primitive group restricts to the regular character of that subgroup. See Listing 2 below for the function definition of ``IsRegularCharacter.''}, label={lst:gap_ids}]
primGrps := AllPrimitiveGroups(NrMovedPoints,prime_power);
BGroups := [];
		
for G in primGrps do
		
	# Ensures G is not doubly transitive
	if not IsTransitive(Stabilizer(G,1)) then
	 	# Conjugates of regular subgroups are regular.
		# It suffices to restrict to a Sylow subgroup.
		P := SylowSubgroup(G,prime);
		chi := NaturalCharacter(P);	
		classes := Filtered(
		ConjugacyClassesSubgroups(P), 
		c -> Size(Representative(c)) = prime_power );
		reps := List( classes, Representative );
			
		for Hk in reps do
			gid := IdGroup(Hk);
			
			if (not (gid in BGroups)) and IsRegularCharacter(Hk, 	
			RestrictedClassFunction(chi,Hk)) then
			
			Add(BGroups, gid);
			Print("Group with ID ", 
			IdGroup(Hk), 
			" and struct. desc. ", 
			StructureDescription(Hk), 
			" found in ", 	
			StructureDescription(G), ".\n");
fi; od; fi; od;
\end{lstlisting}

\begin{lstlisting}[caption={This is a GAP function definition which determines whether a character ``chi'' of the group ``G'' is the regular character. This is used in the code of Listing 2 to test the regularity of a permutation group.}]
IsRegularCharacter := function( G, chi )
local i;

if not (chi[1] = Size(G)) then
	return false;
fi;

for i in [2..Length(chi)] do
	if not (chi[i] = 0) then
		return false;
	fi;
od;

return true;
end;
\end{lstlisting}

\end{document}